\documentclass[a4paper,11pt]{article}

\usepackage[top=3.0cm, bottom=3.0cm, inner=3.0cm, outer=3.0cm, includefoot]{geometry}

\usepackage[toc,page]{appendix}

\usepackage[latin1]{inputenc}
\usepackage[T1]{fontenc}
\usepackage{verbatim}

\usepackage{geometry}
\usepackage{amssymb}
\usepackage{amsmath}
\usepackage{graphicx}
\usepackage{amsthm}

\usepackage{color}

\usepackage{enumerate}
\usepackage{tikz}
\usetikzlibrary{arrows}
\usetikzlibrary{decorations.markings}

\newcommand{\IR}{{\mathbb{R}}}

\newcommand{\IN}{{\mathbb{N}}}

\newcommand{\eChar}{\begin{enumerate}[(i)]}
\newcommand{\eCharR}{\begin{enumerate}[(a)]}
\newcommand{\eBr}{\begin{enumerate}[(1)]}

\newcommand{\AT}{{\mathcal{AT}}}
\newcommand{\KK}{{\mathcal{K}}}
\newcommand{\Id}{{\rm{Id}}}

\title
{
Curvature calculations for antitrees
}

\author{David Cushing, Shiping Liu, Florentin M\"unch, Norbert Peyerimhoff}
\date{\today}

\theoremstyle{plain}
\newtheorem{lemma}{Lemma}[section]
\newtheorem{theorem}[lemma]{Theorem}

\newtheorem{corollary}[lemma]{Corollary}

\theoremstyle{definition}

\newtheorem{definition}{Definition}[section]

\newtheorem{rem}[lemma]{Remark}

\numberwithin{equation}{section}

\begin{document}

\maketitle

\begin{abstract}
  In this article we prove that antitrees with suitable growth properties
  are examples of infinite graphs exhibiting strictly positive
  curvature in various contexts: in the normalized and non-normalized
  Bakry-{\'Emery} setting as well in the Ollivier-Ricci curvature
  case. We also show that these graphs do not have global positive
  lower curvature bounds, which one would expect in view of discrete
  analogues of the Bonnet-Myers theorem. The proofs in the different
  settings require different techniques.
\end{abstract}

\section{Introduction and results}

The main protagonists in this article are \emph{antitrees}. While
these examples had been studied already in 1988, they were given the name
\emph{antitree} in talks by Radoslaw Wojciechowsi
around 2010. A proper definition of antitrees, in their most general
form, appeared first in \cite{KLW13}. Like in the case of a tree, the
vertices of an antitree are partitioned in generations $V_i$ with the
first generation $V_1$ called its \emph{root set}. While trees are
connected graphs with as few connections as possible between
subsequent generations, antitees have the maximal number of
connections. More precisely, antritrees are simple (i.e., no loops and no multiple edges), connected graphs such that
\begin{itemize}
\item[(i)] any root vertex $x \in V_1$ is connected to all vertices in
  $V_2$, and no vertices in $V_k$, $k \ge 3$,
\item[(ii)] any vertex $x \in V_k$, $k \ge 2$, is connected to all vertices
in $V_{k-1}$ and $V_{k+1}$, and no vertices in $V_l$, $|k-l| \ge 2$.
\end{itemize}
Note that this definition allows for the possibility of edges between
vertices of the same generation. We will refer to such edges as
\emph{spherical edges}. Edges between vertices of different
generations are called \emph{radial edges}. Any radial or spherical
edge incident to a vertex in $V_1$ is called \emph{radial} or
\emph{spherical root-edge}, respectively. All other edges are called
\emph{inner edges}.

Antitrees are particularly interesting examples with regards to
stochastic completeness. Section \ref{sec:hist}, provided by Radoslaw
Wojciechowki, gives a more in-depth look at the history of antitrees.
In this article, we investigate curvature properties of
antitrees. Relations between curvature asymptotics and stochastic
completeness were investigated recently in \cite{HuaMunch2017} in the
Bakry-\'Emery setting and in \cite{MunchWoj2017} in the Ollivier-Ricci
curvature setting.

For our curvature considerations, we consider only antitrees where the
induced subgraph of any one generation $V_k$ is complete, i.e., any
two vertices in the same generation are neighbours. For any given
finite or infinite sequence $(a_k)_{1 \le k \le N}$,
$N \in \IN \cup \{\infty\}$, the corresponding unique such antitree
with $|V_k| = a_k$ for all $1 \le k \le N$ is denoted by
$\AT((a_k))$. Note that in the case of a finite antitree, that is
$N < \infty$, (ii) has to be understood in the case $k = N$ that any
vertex $x \in V_N$ is connectd to all vertices in $V_{N-1}$. Later in
this introduction, we will only present results for infinite antitrees
but, since curvature is a local notion, we need only investigate
curvatures of suitable finite antitrees for the proofs. 

\begin{figure}[h]
\begin{center}
\includegraphics[width=0.6\textwidth]{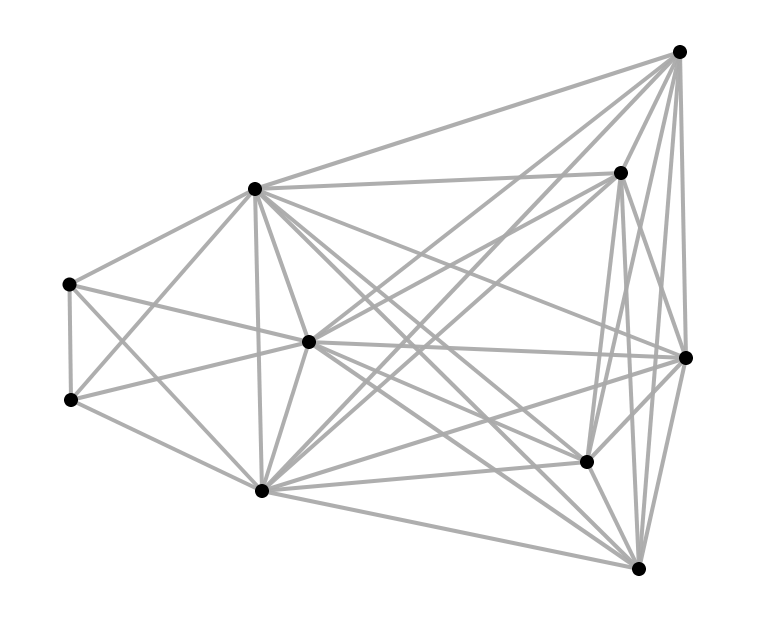}
\end{center}
\caption{The antitree $\mathcal{AT}((2,3,5))$ \label{figure:at235}}
\end{figure}

Two particular curvature notions on graphs have been studied
actively in recent years:
\begin{itemize}
\item \emph{Bakry-\'Emery curvature} taking values on the vertices and based on
Bochner's formula with respect to a suitable graph Laplacian,
\item \emph{Ollivier-Ricci curvature} taking values on the edges and based on
  optimal transport of lazy random walks.
\end{itemize}
Basic graph theoretical notions are introduced in Section
\ref{subsec:basics} and precise definitions of these curvature
concepts are given in Sections \ref{section:BEcurv} and
\ref{section:ORcurv}, respectively.

For both curvature notions there are graph theoretical analogues of
the fundamental Bonnet-Myers Theorem for Riemannian manifolds with
strictly positive Ricci curvature bounded away from zero.

Let us first consider Bakry-{\'E}mery curvature. Generally, on a
combinatorial graph $G= (V,E)$ with vertex set $V$ and edge set $E$,
the graph Laplacian on functions $f: V \to \IR$ is of the form
\begin{equation} \label{eq:lap} 
\Delta f(x) = \frac{1}{\mu(x)} \sum_{y \sim x} (f(y)-f(x)), 
\end{equation}
with a vertex measure $\mu: V \to (0,\infty)$. In this article, we
consider two specific choices of vertex measures:
\begin{itemize}
\item $\mu \equiv 1$, which we refer to as the \emph{non-normalized case},
\item $\mu(x) = d_x$ (the vertex degree of $x \in V$), which we refer
  to as the \emph{normalized case}.
\end{itemize} 
The corresponding discrete Bonnet-Myers theorems in both settings are
as follows:

\begin{theorem}[see \cite{LiuMP2016}] Let $G=(V,E)$ be a
  connected graph satisfying $CD(K,\infty)$ for some $K > 0$ in the
  \emph{non-normalized case} and $d_x \le D$ for all $x \in V$ and
  some finite $D$. Then $G$ is a finite graph and, furthermore,
  $$ {\rm{diam}}(G) \le \frac{2D}{K}. $$
\end{theorem}

\begin{theorem}[see \cite{LiuMP2016}] \label{thm:discBM-norm} Let
  $G=(V,E)$ be a connected graph satisfying $CD(K,\infty)$ for some
  $K > 0$ in the \emph{normalized case} (possibly of unbounded vertex
  degree). Then $G$ is a finite graph and, furthermore,
  $$ {\rm{diam}}(G) \le \frac{2}{K}. $$
\end{theorem}

Ollivier-Ricci curvature depends upon an idleness parameter
$p \in [0,1]$ describing the laziness of the associated random
walk. Here, the discrete Bonnet-Myers theorem takes the following
form:

\begin{theorem}[see \cite{Oll2009}] Let $G=(V,E)$ be a connected graph
  satisfying $\kappa_p(x,y) \ge K > 0$ for all $x \sim y$ and a fixed
  idleness $p \in [0,1]$.  Then $G$ is a finite graph and,
  furthermore, 
  \begin{equation} \label{eq:DBM} 
 {\rm{diam}}(G) \le \frac{2(1-p)}{K}. 
  \end{equation}
\end{theorem}

These results give rise to the following natural questions:
\begin{itemize}
\item Do there exist examples of infinite connected graphs with
  strictly positive curvature? (That is, relaxing the condition of a
  uniform strictly positive lower curvature bound.)
\item In the non-normalized case, doe there exist an infinite
  connected graphs satisfying $CD(K,\infty)$ for $K > 0$ of unbounded
  vertex degree?
\end{itemize}

This paper provides a positive answer to the first question. In fact,
we show that antitrees $\AT((a_k))$ with suitable growth properties of
the infinite sequence $(a_k)$ have strictly positive curvature for all
curvature notions mentioned above. More precisely, we have the following
in the Bakry-{\'E}mery curvature case:

\begin{theorem} In both the normalized and non-normalized setting, the
  infinite antitree $\AT((k))$ satisfies $CD(K_x,\infty,x)$
  for all vertices $x$ with a family of constants $K_x >
  0$ depending only on the generation of $x$. Furthermore, 
  $$ \liminf_{k \to \infty,\, x \in V_k} K_x = 0. $$
\end{theorem}

\begin{rem} \label{rem:bakem_examples}
  In fact, the method of proof relies on some Maple calculations which
  can be extended to also provide the following results (without going
  into the details):
  \begin{itemize}
  \item[(i)] \emph{Linear growth:} The same curvature results hold true for
    the infinite antitrees\\ $\AT((1+(k-1)t))$ with arbitrary
    $t \in \IN$.
  \item[(ii)] \emph{Exponential growth:} The same curvature results
    hold true for the infinite antitree $\AT((2^{k-1}))$ in the
    normalized case and fails to satisfy $CD(0,\infty)$ in the
    non-normalized case.
  \end{itemize}
\end{rem}

Due to Bakry-{\'E}mery curvature being a local property, in order to
calculate the curvatures $\KK_{G,x}(\infty)$ of vertices $x$ \emph{in the
first two generations} of $G = \AT((2^{k-1}))$ as defined later in
\eqref{eq:KinfCD}, it is sufficient to consider the graph presented in
Figures \ref{figure:cur_norm} and \ref{figure:cur_non_norm} (spherical
edges of $2$-spheres around a vertex do not contribute to the
curvature, see \cite{CLP2016}). These figures are in agreement with
the statements in Remark \ref{rem:bakem_examples}(ii).

\begin{figure}[h]
\begin{minipage}[t]{0.5\linewidth}
\centering
\includegraphics[width=\textwidth]{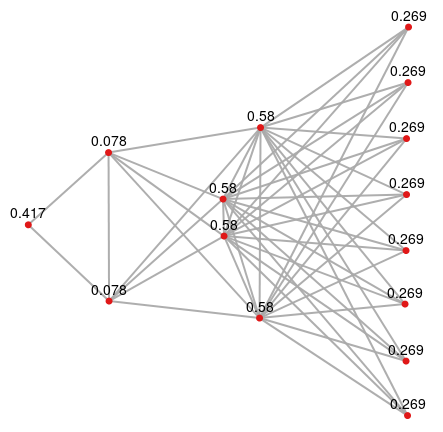}
\caption{Normalized curvature $\KK_{G,x}(\infty)$ \label{figure:cur_norm}}
\end{minipage}
\hfill
\begin{minipage}[t]{0.56\linewidth}
\centering
\includegraphics[width=.85\textwidth]{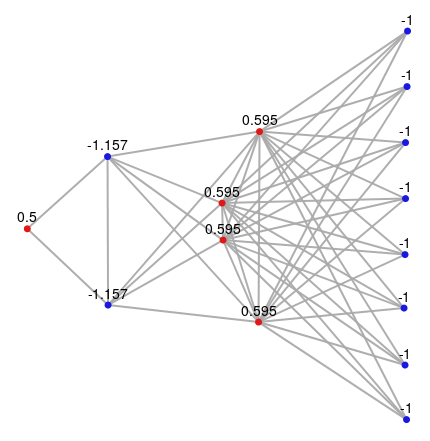}
\caption{Non-normalized curvature $\KK_{G,x}(\infty)$ \label{figure:cur_non_norm}}
\end{minipage}
\end{figure}

Now we consider Ollivier-Ricci curvature. Here our main result is the
following:

\begin{theorem}
  Let $G = \AT((a_k))$ be an infinite antitree with $1 = a_1$ and
  $a_{k+1} \ge a_k$ for all $k \in \IN$ and $x,y$ be neighbouring
  vertices in $G$. 
  \begin{itemize}
  \item \emph{Radial root edges:} If $x \in V_1$ and $y \in V_2$:
  $$ \kappa_p(x,y) = \begin{cases} \frac{a_2-1}{a_2+a_3}+\frac{a_2+2a_3+1}{a_2+a-3}p, & \text{if $p \in \left[ 0, \frac{1}{a_2+a_3+1}\right]$,} \\[.2cm]
\frac{a_2+1}{a_2+a_3}(1-p), & \text{if $p \in \left[\frac{1}{a_2+a_3+1},1\right]$.} \end{cases} $$
\item \emph{Radial edges:} If $x \in V_k$ and $y \in V_{k+1}$,
  $k \ge 2$, $p \in [0,1]$:
  $$ \kappa_p(x,y) = \left( \frac{2a_k+a_{k+1}-1}{a_k+a_{k+1}+a_{k+2}-1} -
\frac{2a_{k-1}+a_k-1}{a_{k-1}+a_k+a_{k+1}-1} \right) (1-p). $$
\item \emph{Spherical edges:} If $x,y \in V_k$, $x \neq y$, $k \ge 2$:
  $$ \kappa_p(x,y) = \begin{cases} \frac{a_{k-1}+a_k+a_{k+1}-2}{a_{k-1}+a_k+a_{k+1}-1} + \frac{a_{k-1}+a_k+a_{k+1}}{a_{k-1}+a_k+a_{k+1}-1}p, &
\text{if $p \in \left[0, \frac{1}{a_{k-1}+a_k+a_{k+1}}\right]$,} \\[.2cm]
\frac{a_{k-1}+a_k+a_{k+1}}{a_{k-1}+a_k+a_{k+1}-1}(1-p), & 
\text{if $p \in \left[\frac{1}{a_{k-1}+a_k+a_{k+1}},1\right]$.} \end{cases} $$
  \end{itemize}
\end{theorem}

Let us consider special cases:



\begin{corollary}[Linear growth] Let $G = \AT((1+(k-1)t))$, $ t \in \IN$ arbitrary. Then
$$
\kappa_0(x,y) = \begin{cases}
\frac{t}{3t+2} & \text{for $x \in V_1$, $y \in V_2$,} \\
\frac{6t^2}{(3kt+2)(3kt+2-3t)} & \text{for $x \in V_k$, $y \in V_{k+1}$,} \\ 
1 - \frac{1}{3kt+2-3t} & \text{for $x,y \in V_k$, $x \neq y$, $k \ge 2$.} 
\end{cases}
$$
In particular, $\kappa_0$ of radial edges decays asymptotically like
$\frac{2}{3k^2}$ as $k \to \infty$.
\end{corollary}


\begin{corollary}[Exponential growth] We have for $G = \AT((r^{k-1})$,
  $r \in \IN$:
$$
\kappa_0(x,y) = \begin{cases}
\frac{r-1}{r(r+1)} & \text{for $x \in V_1$, $y \in V_2$,} \\
\frac{(r-1)^2(r+1)r^{k-2}}{(r^{k}+r^{k-1}+r^{k-2}-1)(r^{k+1}+r^k+r^{k-1}-1)} & \text{for $x \in V_k$, $y \in V_{k+1}$,} \\ 
1 - \frac{1}{r^k+r^{k-1}+r^{k-2}-1} & \text{for $x,y \in V_k$, $x \neq y$, $k \ge 2$.} 
\end{cases}
$$
In particular, $\kappa_0$ of radial edges decays asymptotically like
$\frac{1}{r^k}$ as $k \to \infty$.
\end{corollary}

\begin{rem} Note that for any finite sequence $(a_k)_{1 \le k \le N}$,
  $N \ge 2$, with $1 = a_1$ and $a_{k+1} \ge a_k$ for all
  $1 \le k \le N$, we can find a large enough $a_{N+1} \ge a_N$ such that
  $\kappa_0(x,y) < 0$ for $x \in V_{N-1}$ and $y \in V_N$.
\end{rem}

The paper is organised as follows: We start with some historical
comments on antitrees in Section \ref{sec:hist} which was provided by
Rados\l{}aw Wojciechowski. Section \ref{sec:basics} introduces the
readers into Bakry-{\'E}mery curvature and Ollivier-Ricci
curvature. The following two Sections \ref{sec:becalcs} and
\ref{sec:ollcalcs} present the concrete curvature investigations in
both settings. The Appendices \ref{app:thm:atabcde-BE},
\ref{app:thm:atbcde-BE}, and \ref{app:thm:curv0BE} provide the Maple
code used for the results in Section \ref{sec:becalcs}.

{\bf{Acknowledgement:}} We are grateful to Radoslaw Wojciechowski,
Matthias Keller, and Jozef Dodziuk for providing useful information on
antitrees. Some figures in this article are based on the curvature
calculator by David Cushing and George Stagg (see \cite{CKLLS2017}).

\section{A (partial) history of antitrees} \label{sec:hist}

To our knowledge, the first known appearance of an antitree is the
case of $|S_r| = r+1$ in the article of Dodziuk and Karp
\cite{DK88}. They study the normalized Laplacian $\Delta$ and give
conditions for transience of the simple random walk in terms of
$r \Delta r$ where $r$ is the distance to a vertex. It appears in
\cite[Example 2.5]{DK88} as a case of a transient graph with bottom of
the spectrum $0$ whose Green's function decays like $1/r$. The same
antitree appears in the article of Weber \cite{Web10}. Weber extends
the result of Dodziuk and Mathai \cite{DM06} concerning the stochastic
completeness of the semigroup associated to the non-normalized
Laplacian $\Delta$. Indeed, Dodziuk/Mathai prove stochastic
completeness in the case of bounded vertex degree. Weber improves this
result to give stochastic completeness in the case of $\Delta r \ge K$
for some constant $K$. The antitree mentioned above is then given as
an example of a graph whose vertex degree is unbounded but which
satisfies $\Delta r \ge K$, see \cite[Figure 1, p. 156]{Web10}. The
general case of antitrees with arbitrary spherical growth
$|S_r| = f(r)$ where $f$ is any natural number valued function is
considered in \cite[Example 4.11]{Woj11}. There it is shown that
antitrees are stochastically complete if and only if
$$ \sum_r \frac{\sum_{k=0}^r f(k)}{f(r)f(r+1)} = \infty. $$
This is used to give a counterexample to a direct analogue to
Grigor'yan's result for stochastic completeness of manifolds (see
\cite{Gri99}). Indeed, Grigor'yan's result says that any
stochastically incomplete manifold must have superexponential volume
growth while the result above gives stochastically incomplete graphs
which have only polynomial volume growth when the combinatorial graph
metric is used. These examples give the smallest such examples in the
combinatorial graph metric by a result of Huang, Grigor'yan and
Masamune \cite[Theorem 1.4]{GHM12}, where the example (and name) of
antitrees also appears. This might be the first time in print that the
name is used and they refer to them as the ''\emph{antitree} of
Wojciechowski".  A proper definition with the name of antitree first
appears in \cite[Definition 6.3]{KLW13}. Here the result on stochastic
completeness is generalized to all weakly spherically symmetric graphs
of which the antitrees are but an example. Furthermore, it is shown
that the non-normalized Laplacian $\Delta$ on any such stochastically
incomplete antitree has positive bottom of the spectrum, see \cite[Corollary
6.6]{KLW13}. This gives a counterexample to a direct analogue to a theorem of
Brooks \cite{Bro81} which states that the bottom of the spectrum of
the Laplacian on any manifold with subexponential volume growth is
zero. This sparked an interest in applying intrinsic metrics as
defined by Frank, Lenz and Wingert in \cite{FLW14} to study the
question involving volume growth on graphs of unbounded vertex
degree. In particular, the analogue to Grigor'yan's theorem was first
proven in \cite{Fol14} (see also \cite{Hua14} for an analytic proof)
while the analogue to Brooks' theorem was shown in \cite{HKW13}. Since
then, antitrees appear in a variety of places. Their spectral theory
is thoroughly analyzed by Breuer and Keller in \cite{BK13}. Here it
should be noted that the spectrum consists mainly of eigenvalues with
compactly supported eigenfunctions and a further spectral component
which can be singular continuous in certain cases. Antitrees are also
used as a counterexample to a conjecture presented by Golenia and
Schumacher in \cite{GS11} concerning the deficiency indices of the
adjacency matrix, see \cite{GS13}. They are also used to show the
utility of the new bottom of the spectrum estimate for a Cheeger
constant involving intrinsic metrics in \cite{BKW15}.

\section{Definitions and notations}
\label{sec:basics}

\subsection{Basic graph theoretical notations}
\label{subsec:basics}

Let $G=(V,E)$ be a locally finite connected \emph{simple}
combinatorial graph (that is, no loops and no multiple edges) with
vertex set $V$ and edge set $E$. For any $x,y \in V$ we write
$x \sim y$ if $\{x,y\} \in E$. The \emph{degree of a vertex} $x \in V$
is denoted by $d_x$. Let $d: V \times V \to \IN \cup \{0\}$ be the
\emph{combinatorial distance function}, i.e., $d(x,y)$ is the length
of the shortest path from $x$ to $y$. For $x \in V$, the \emph{combinatorial
spheres} and \emph{balls} of radius $r \ge 0$ around $x$ are denoted by
\begin{eqnarray*}
S_r(x) &=& \{ y \in V \mid d(x,y) = r \}, \\
B_r(x) &=& \{ y \in V \mid d(x,y) \le r \},
\end{eqnarray*}
respectively. The \emph{diameter} of $G$ is
defined as
$$ 
{\rm{diam}}(G) = \sup\{d(x,y) \mid x,y \in V\} \in \IN \cup \{0,\infty\}. 
$$

\subsection{Bakry-\'Emery curvature} \label{section:BEcurv}

As mentioned before, this curvature notion is rooted on Bochner's
formula using a Laplacian operator leading to the curvature-dimension
inequality (CD-inequality for short). This approach was pursued by
Bakry-\'Emery \cite{BE85} via an elegant $\Gamma$-calculus and lead to
a substitute of the lower Ricci curvature bound of the underlying
space for much more general settings. (Some further information on the
Bochner approach can be found, e.g., in \cite[Remark 1.3]{CLP2016}).

Recall the definition \eqref{eq:lap} of the normalized
($\mu(x) = d_x$) and non-normalized Laplacian ($\mu \equiv 1$) from
the Introduction. Such a choice of Laplacian leads to the following
operator $\Gamma$ for all $f,g: V \to \IR$:
\begin{eqnarray*}
\Gamma(f,g)(x)&=&\frac{1}{2}(\Delta(fg)-f\Delta g-g\Delta f)(x) \\
&=& \frac{1}{2\mu(x)} \sum_{y \sim x} (f(y)-f(x))(g(y)-g(x)).
\end{eqnarray*}
For simplicity, we always write $\Gamma(f):=\Gamma(f,f)$. Iterating
$\Gamma$, we can define another operator $\Gamma_2$, given by
$$\Gamma_2(f,g)(x)=\frac{1}{2}(\Delta\Gamma(f,g)-\Gamma(f,\Delta
g)-\Gamma(g,\Delta f))(x).$$ 
Again, we abbreviate $\Gamma_2(f)=\Gamma_2(f,f)$. The Bakry-{\'E}mery
curvature is defined via these operators in the following way.

\begin{definition} Let $K\in \IR$ and $N\in (0,\infty]$.
  \begin{enumerate}[(i)]
  \item The pointwise curvature dimension condition $CD(K,N,x)$ for
    $x\in V$ is defined by
    $$\Gamma_2(f)(x)\geq K\Gamma(f)(x)+\frac 1N(\Delta
    f)^2(x),\,\,\,\text{for any }\,\,\,f: V\to \mathbb{R}.$$
  \item The global curvature dimension condition $CD(K,N)$ holds if and only if
    $CD(K,N,x)$ holds for any $x\in V$.
  \item For any $x\in V$, we define
    \begin{equation} \label{eq:KinfCD}
    {\mathcal{K}}_{G,x}(N):=\sup\{K\in\IR\mid CD(K,N,x)\}.
    \end{equation}
  \end{enumerate}
\end{definition}

In this article, we are only concerned with $\infty$-curvature, that
is, $N = \infty$. Following \cite[Prop. 2.1]{CLP2016}, the condition
$CD(K,\infty,x)$ is equivalent to
\begin{equation} \label{eq:BEcurv} 
\Gamma_2(x) \ge K \Gamma(x), 
\end{equation}
where $\Gamma_2(x)$ and $\Gamma(x)$ are symmetric matrices of the
corresponding quadratic forms evaluated at $x \in V$. Since only local
information needs to be taken into account, they are of size
$|B_2(x)| \times |B_2(x)|$ and $|B_1(x)| \times |B_1(x)|$,
respectively, and to make sense of \eqref{eq:BEcurv} the smaller size
matrix must be padded with $0$ entries. For more information in the
non-normalized case, see \cite[Sections 2.1-2.3]{CLP2016}.
The entries of these matrices in the general weighted case
are explicitly given in \cite[Section 12]{CLP2016}. (Note that for the
context of this article, the edge weights $w: E \to [0,\infty)$ take
only values $0,1$ and reflect adjacency of vertices and the vertex
measure $\mu: V \to (0,\infty)$ will only correspond to the normalized
and non-normalized cases.)

The main tool to prove strictly positive curvature is \cite[Corollary
2.7]{CLP2016}, that is, the following properties are equivalent:
\begin{itemize} 
\item $\Gamma_2(x)$ is positive semidefinite with one-dimensional kernel,
\item ${\mathcal{K}}_{G,x}(\infty) > 0$.
\end{itemize}
\cite[Corollary 2.7]{CLP2016} covers only the non-normalized case but
one can easily check that the equivalence holds also in the setting of
general vertex measures.

\subsection{Ollivier-Ricci curvature} \label{section:ORcurv}

As mentioned before, Ollivier-Ricci curvature is based on optimal
transport.  Ollivier-Ricci curvature was introduced in
\cite{Oll2009}. A fundamental concept in optimal transport is the
Wasserstein distance between probability measures.

\begin{definition}
Let $G = (V,E)$ be a locally finite graph. Let $\mu_{1},\mu_{2}$ be two probability measures on $V$. The {\it Wasserstein distance} $W_1(\mu_{1},\mu_{2})$ between $\mu_{1}$ and $\mu_{2}$ is defined as
\begin{equation} \label{eq:W1def}
W_1(\mu_{1},\mu_{2})=\inf_{\pi} \sum_{y\in V}\sum_{x\in V} d(x,y)\pi(x,y),
\end{equation}
where the infimum runs over all transportation plans $\pi:V\times  V\rightarrow [0,1]$ satisfying
$$\mu_{1}(x)=\sum_{y\in V}\pi(x,y),\:\:\:\mu_{2}(y)=\sum_{x\in V}\pi(x,y).$$
\end{definition}

The transportation plan $\pi$ moves a mass
distribution given by $\mu_1$ into a mass distribution given by
$\mu_2$, and $W_1(\mu_1,\mu_2)$ is a measure for the minimal effort
which is required for such a transition.
\\
\\
If $\pi$ attains the infimum in \eqref{eq:W1def} we call it an {\it
  optimal transport plan} transporting $\mu_{1}$ to $\mu_{2}$.
\\
\\
We define the following probability distributions $\mu_x$ for any
$x\in V,\: p\in[0,1]$:
$$\mu_x^p(z)=\begin{cases}p,&\text{if $z = x$,}\\
\frac{1-p}{d_x},&\text{if $z\sim x$,}\\
0,& \mbox{otherwise.}\end{cases}$$

\begin{definition}
The $ p-$Ollivier-Ricci curvature on an edge $x\sim y$ in $G=(V,E)$ is
$$\kappa_{ p}(x,y)=1-W_1(\mu^{ p}_x,\mu^{ p}_y),$$
where $p \in [0,1]$ is called the {\it idleness}.

The Ollivier-Ricci curvature introduced by Lin-Lu-Yau in
\cite{LLY11}, is defined as
$$\kappa_{LLY}(x,y) = \lim_{ p\rightarrow 1}\frac{\kappa_{ p}(x,y)}{1- p}.$$
\end{definition}

A fundamental concept in the optimal transport theory and vital to our work is Kantorovich duality. First we recall the notion
of 1--Lipschitz functions and then state Kantorovich duality.

\begin{definition}
Let $G=(V,E)$ be a locally finite graph, $\phi:V\rightarrow\mathbb{R}.$ We say that $\phi$ is $1$-Lipschitz if 
$$|\phi(x) - \phi(y)| \leq d(x,y)$$
for all $x,y\in V.$ Let \textrm{1--Lip} denote the set of all $1$--Lipschitz functions. 
\end{definition}

Note that, by triangle inequality, $\phi$ is $1$--Lipschitz iff
$|\phi(x)-\phi(y)| \le 1$ for all paris $x \sim y$.

\begin{theorem}[Kantorovich duality]\label{Kantorovich}
Let $G = (V,E)$ be a locally finite graph. Let $\mu_{1},\mu_{2}$ be two probability measures on $V$. Then
$$W_1(\mu_{1},\mu_{2})=\sup_{\substack{\phi:V\rightarrow \mathbb{R}\\ \phi\in \textrm{\rm{1}--{\rm Lip}}}}  \sum_{x\in V}\phi(x)(\mu_{1}(x)-\mu_{2}(x)).$$
\\
\\
If $\phi \in \textrm{1--Lip}$ attains the supremum we call it an \emph{optimal Kantorovich potential} transporting $\mu_{1}$ to $\mu_{2}$.
\end{theorem}

The following result on some properties of $p \mapsto \kappa_p(x,y)$
for $x \sim y$ and its consequences was useful in our curvature
considerations.

\begin{theorem}[see \cite{BCLMP}]
  Let $G=(V,E)$ be a locally finite graph. Let $x,y\in V$ with
  $x\sim y.$ Then the function $p \mapsto \kappa_{p}(x,y)$ is concave
  and piecewise linear over $[0,1]$ with at most $3$ linear
  parts. Furthermore $\kappa_{p}(x,y)$ is linear on the intervals
  \begin{equation*}
    \left[0,\frac{1}{{\rm{lcm}}(d_{x},d_{y})+1}\right]\:\:\: {\rm and} \:\:\:\left[\frac{1}{\max(d_{x},d_{y})+1},1\right].
  \end{equation*}
  Thus, if we have the further condition $d_{x}=d_{y}$, then
  $\kappa_{p}(x,y)$ has at most two linear parts.
\end{theorem}

\section{Bakry-\'Emery curvature of antitrees}
\label{sec:becalcs}

Let us first introduce some notation and a useful general fact (Lemma
\ref{lem:usefact} below). The identity matrix of size $d$ is denoted
by ${\rm Id}_d$ and the all-zero and all-one matrix of size
$d_1 \times d_2$ is denoted by $0_{d_1,d_2}$ and $J_{d_1,d_2}$,
respectively. Moreover, if $d_1 = d_2$, we use the notation
$J_{d_1} = J_{d_1,d_1}$, and if $d_2 = 1$, we use the notation
${\bf 1}_{d_1}$ for the all-one column vector of size $d_1$. Moreover,
the standard base of column vectors in $\IR^N$ is denoted by
$e_1,\dots,e_N$.

\begin{lemma} \label{lem:usefact} Let $d_1,\dots,d_r \in \IN$ and
  $A = (A_{ij})_{1 \le i,j \le r}$ be a symmetric matrix, where the
  $A_{ij}$ are block matrices of size $d_i \times d_j$ with
  $A_{ji} = A_{ij}^\top$. Assume there exist constants
  $\alpha_i, \beta_i \in \IR$ and $\gamma_{ij} = \gamma_{ji} \in \IR$
  such that, for $1 \le i,j \le r$, $j \neq i$,
  $$ A_{ii} = \alpha_i {\rm Id}_{d_i} + \beta_i J_{d_i} $$
  and
  $$ A_{ij} = \gamma_{ij} J_{d_i,d_j}. $$ 
  Let $A_{\rm{red}} = (a_{ij})_{1 \le i,j \le r}$ be the $r \times r$-matrix
  given by $a_{ij} = {\bf 1}_{d_i}^\top A_{ij} {\bf 1}_{d_j}$, i.e.,
  for $i \neq j$,
  \begin{eqnarray*}
  a_{ii} &=& \alpha_i d_i + \beta_i d_i^2, \\
  a_{ij} &=& \gamma_{ij} d_i d_j.
  \end{eqnarray*}  
  For any vector $w = (w_1,\dots,w_r )^\top \in \IR^r$ let
  $$ \widehat w := (w_1 {\bf 1}_{d_1}^\top, \dots, w_r {\bf 1}_{d_r}^\top )^\top
  \in \IR^{d} $$
  with $d = \sum_{j=1}^r d_j$. Then we have the following two facts:
  \begin{itemize}
  \item[(a)] For every $d_i \ge 2$, the $(d_i- 1)$-dimensional space
  $$ E_i = \left\{ \sum_{j=1}^{d_i} c_j e_{j+d} \mid \sum_{j=1}^{d_i}
    c_j = 0 \right\} $$ 
  with $d = \sum_{j=1}^{i-1} d_j$ consists of eigenvectors to the
  eigenvalue $\alpha_i$.
  \item[(b)] For any $w \in \IR^r$, the corresponding vector $\widehat w$
  is orthogonal to all spaces $E_i$ in (a) and we have
  $$ {\widehat w}^\top A \widehat\, w = w^\top A_{\rm{red}}\, w. $$
  \end{itemize}
\end{lemma}

The proof of this lemma is a straightforward calculation and left to
the reader.

Now we start with our Bakry-{\'E}mery curvature considerations for
antitrees. Due to localness of the Bakry-{\'E}mery curvature notion,
we only need to consider ${\mathcal K}_{G,x}(\infty)$ for
\begin{itemize}
\item[(i)] a vertex $x \in V_3$ in the finite antitree
  $\AT((a,b,c,d,e))$, 
\item[(ii)] a vertex $x \in V_2$ in the finite antitree $\AT((b,c,d,e))$, and
\item[(iii)] a vertex $x \in V_1$ in the finite antitree $\AT((c,d,e))$.
\end{itemize}
The relevant results are given in the following theorems.

\begin{theorem} \label{thm:atabcde-BE}
  Let $x \in V_3$ be a vertex of the finite antitree
  $G = \AT((a,b,c,d,e))$. If
  $$ a = n, \,\, b = n+1, \,\, c = n+2, \,\, d = n+3, \,\, \text{and} \,\, 
  e = n+4, $$
  we have in both the normalized and non-normalized case:
  \begin{equation} \label{eq:poscur}
  {\mathcal K}_{G,x}(\infty) > 0.
  \end{equation}
\end{theorem}

\begin{proof}
  In this proof, we will keep the values $a,b,c,d,e$ general as long
  as possible and only specify them towards the end of the proof. Let
  $G = \AT((a,b,c,d,e))$, $1 \le a \le b < c \le d \le e$ and
  $x \in V_3$. To cover simultaneously both the normalized and
  non-normalized setting, we choose
  $$ \epsilon_- = \frac{\mu(x)}{\mu(y_-)} - 1, \quad \epsilon_+ =
  \frac{\mu(x)}{\mu(y_+)} - 1, $$ 
  where $y_- \in V_2$ and $y_+ \in V_4$. (Note that $\mu(z)$ depends
  only the generation of $z$.) Using the results in \cite[Section
  12]{CLP2016}, a tedious but straightforward calculation shows the
  following: The matrix $A = 4 \mu(x)^2 \Gamma_2(x)$ is of the
  following block structure $A = (A_{ij})_{1 \le i,j \le 6}$ where the
  blocks correspond to an ordering of $B_2(x)$ into the vertex sets
  $\{x\}, V_3 \backslash \{x\}, V_4, V_2, V_5, V_1$:
  \begin{eqnarray*}
    A_{11} &=& d_x(d_x+3) + 3 b \epsilon_- + 3 d \epsilon_+, \\
    A_{12} &=& (-(d_x+3) + b \epsilon_- + d \epsilon_+)J_{1,c-1}, \\
    A_{13} &=& (-(d_x+3+e)-(2+c+e)\epsilon_+)J_{1,d}, \\
    A_{14} &=& (-(d_x+3+a)-(2+a+c)\epsilon_-)J_{1,b}, \\
    A_{15} &=& (d+d \epsilon_+)J_{1,e}, \\
    A_{16} &=& (b+b \epsilon_-)J_{1,a}, \\
    A_{22} &=& (3(d_x+1)+b\epsilon_-+d\epsilon_+)\Id_{c-1}-2J_{c-1}, \\
    A_{23} &=& -(2+2\epsilon_+)J_{c-1,d}, \\
    A_{24} &=& -(2+2\epsilon_-)J_{c-1,b}, \\
    A_{25} &=& 0_{c-1,e}, \\
    A_{26} &=& 0_{c-1,a}, \\
    A_{33} &=& (-b+3c+3d+3e+(3c+4d+3e)\epsilon_+)\Id_d-(2+4\epsilon_+)J_d, \\
    A_{34} &=& 2J_{d,b}, \\
    A_{35} &=& -(2+2\epsilon_+)J_{d,e}, \\
    A_{36} &=& 0_{d,a}, \\
    A_{44} &=& (3a+3b+3c-d+(3a+4b+3c)\epsilon_-)\Id_b - (2+4\epsilon_-)J_b, \\
    A_{45} &=& 0_{b,e}, \\
    A_{46} &=& -(2+2\epsilon_-)J_{b,a}, \\
    A_{55} &=& (d+d\epsilon_+)\Id_e, \\
    A_{56} &=& 0_{e,a}, \\
    A_{66} &=& (b+b\epsilon_-)\Id_a.
  \end{eqnarray*}
  Let $A_{\rm{red}}$ be the corresponding reduced symmetric
  $6 \times 6$ matrix $A_{\rm{red}} =(a_{ij})_{1 \le i,j \le 6}$, as
  defined in Lemma \ref{lem:usefact}. 

  Recalling the equivalence at the end of Section
  \ref{section:BEcurv}, ${\mathcal{K}}_{G,x}(\infty) > 0$ is
  equivalent to $A$ being positive semidefinite and having
  one-dimensional kernel. Lemma \ref{lem:usefact} provides the
  following eigenvalues and multiplicites of $A$: 
  \begin{itemize}
  \item Since $\epsilon_-, \epsilon_+ > -1$ and $d_x = b+c+d-1$, 
  $$ \alpha_2 = 3(d_x+1_+b\epsilon_-+d\epsilon_+) > 0 $$
  is a positive eigenvalue of multiplicity $c-2 \ge 0$.
  \item Note that in both normalized and non-normalized case we
  have $\epsilon_+ \ge \frac{b+c+d-1}{c+d+e-1} - 1$ and
  \begin{multline*} 
  \alpha_3 = -b+3c+3d+3e+(3c+4d+3e)\epsilon_+ \ge \\ \ge 
  -b-d + \frac{3c+4d+3e}{c+d+e-1} (b+c+d-1) > 0
  \end{multline*}
  is a positive eigenvalue of multiplicity $d-1 \ge 1$.
  \item Note that in both normalized and non-normalized case we have
  $\epsilon_- \ge 0$ and
  $$ \alpha_4 = 3a+3b+3c-d+(3a+4b+3c)\epsilon_- \ge 3a+3b+3c-d > 0$$
  if $d < 3(a+b+c)$. This eigenvalue has multiplicity $b-1 \ge 0$.
  \item Since $\epsilon_-,\epsilon_+ > -1$, 
  $$ \alpha_5 = d+d\epsilon_+ > 0 \quad \text{and} \,\, 
  \alpha_6 = b+b\epsilon_- > 0$$ are both positive eigenvalues of
  multiplicities $e-1 \ge 1$ and $a-1 \ge 0$, respectively.
  \end{itemize}
  Moreover, it is easily checked that $A {\bf 1}_{a+b+c+d+e} = 0$. The
  orthogonal complement of the direct sum of the corresponding
  eigenspaces $E_i$ and $\IR {\bf 1}_{a+b+c+d+e}$ is
  $5$-dimensional and given by $\widehat W = \{ \widehat w \mid w \in W
  \}$, where $(d_1,d_2,d_3,d_4,d_5,d_6) = (1,c-1,d,b,e,a)$ and 
  $$ W := \{ w \in \IR^6, \sum_{i=1}^6 w_i d_i = 0 \}. $$
  Under the assumption $d < 3(a+b+c)$,
  ${\mathcal{K}}_{G,x}(\infty) > 0$ is then equivalent to
  $A \vert_{\widehat W}$ being positive definite, which is equivalent
  to
  \begin{equation} \label{eq:poscurv} 
  {\widehat w}^\top A\, \widehat w = w^\top A_{\rm{red}}\, w > 0 \quad
  \text{for all $w \in W \backslash \{0\}$.}
  \end{equation}
  Now we choose $(a,b,c,d,e) = (n,n+1,n+2,n+3,n+4)$, $n \in \IN$. Then
  we have $d < 3(a+b+c)$ and we consider the characteristic polynomial
  of $A_{\rm{red}}$, which is of the form
  $$ 
  \chi_n(t) = \det(t \Id_6 - A_{\rm{red}}) = 
  t^6 - p_5(n) t^5 + p_4(n) t^4 - p_3(n) t^3 + p_2(n) t^2
  - p_1(n) t,
  $$
  where $p_i(n)$ are polynomials in the variable $n$. (We do not have
  a constant term since $\IR \cdot {\bf 1}_6$ lies in the kernel of
  $A_{\rm{red}}$.) A Maple calculation shows that all the $p_i(n)$ are
  strictly positive for any value of $n \in \IN$ (see Appendix
  \ref{app:thm:atabcde-BE} for more details). This shows that we have
  $\chi_n(t) > 0$ for all $t < 0$, so $A_{\rm{red}}$ is positive
  semidefinite. Since $p_1(n) > 0$, $A_{\rm{red}}$ has a
  one-dimensional kernel $\IR \cdot {\bf 1}_6$.

  Now we can show \eqref{eq:poscurv}: Let
  $w_0 = {\bf 1}_6, w_1, \dots, w_5 \in \IR^6$ be a basis of
  eigenvectors of $A_{\rm{red}}$, i.e.,
  $A_{\rm{red}} w_j = \lambda_j w_j$ with $\lambda_j > 0$ for
  $j \in \{1,\dots,5\}$. Any vector $w \in W \backslash \{0\}$ is of
  the form $w = \sum_{j=0}^5 c_j w_j$ with some $c_{j_0} \neq 0$,
  $j_0 \in \{1,\dots,5\}$, since $w_0 \not\in W$. This implies
  $$ w^\top A_{\rm{red}}\, w = \sum_{j=1}^5 \lambda_j c_j^2 \ge
  \lambda_{j_0} c_{j_0}^2 > 0. $$
\end{proof}

\begin{theorem} \label{thm:atbcde-BE} 
  Let $x \in V_2$ be a vertex of the finite antitree
  $G = \AT((b,c,d,e))$. If $(c,d,e) = (1,2,3)$, we have in both the normalized
  and non-normalized case:
  $$  
  {\mathcal K}_{G,x}(\infty) > 0.
  $$
\end{theorem}

\begin{proof}
  We consider again the matrix $A = 4 \mu(x)^2 \Gamma_2(x)$ and choose
  right from the beginning $(b,c,d,e) = (1,2,3,4)$. It can be checked
  that this time the matrix $A$ is of the form
  $A = (A_{ij})_{1 \le i,j \le 5}$ with $A_{ij}$ as in the previous
  proof and $a=0$. As in the previous proof, we conclude that $A$ has
  eigenvalues $\alpha_3 = 27+30\epsilon_+ > 0$ of multiplicity $2$ and
  $\alpha_5 = 1+\epsilon_+ > 0$ of multiplicity $3$ and that
  $A {\bf 1}_{10}=0$. In this case, $A_{\rm{red}}$ is a symmetric
  $5 \times 5$ matrix and its characteristic polynomial of
  $A_{\rm{red}}$ is (see Maple calculations in Appendix
  \ref{app:thm:atbcde-BE})
  $$ \chi(t) = \det(t\Id_5 - A_{\rm{red}}) = t^5 - \frac{471}{4} t^4 + \frac{118743}{32} t^3 - \frac{593811}{16} t^2 + \frac{3082725}{64} t $$ 
  in the normalized case and
  $$ \chi(t) = t^5 - 132 t^4 + 3684 t^3 - 25632 t^2 + 8640 t $$
  in the non-normalized case. The same arguments as in the previous
  proof show that $A$ is positive semidefinite with one-dimensional
  kernel, that is, ${\mathcal K}_{G,x}(\infty) > 0$.
\end{proof}

\begin{theorem} \label{thm:atcde-BE}
  Let $x \in V_1$ be a vertex of the finite antitree $G =
  \AT(c,d,e)$. If $(c,d,e) = (1,2,3)$, we have in both the normalized
  and non-normalized case:
  $$  
  {\mathcal K}_{G,x}(\infty) > 0.
  $$
\end{theorem}

\begin{proof}
  As in the previous proof, we consider the matrix
  $A = 4 \mu(x)^2 \Gamma_2(x)$ and choose $(c,d,e) = (1,2,3)$. This
  time $A$ is of the form $A = (A_{ij})_{i,j \in I}$ with
  $I = \{1,3,4\}$ and $A_{ij}$ as in the proof of Theorem
  \ref{thm:atabcde-BE} with $a=b=0$. As before, we conclude that $A$
  has a simple eigenvalue $\alpha_3 = 18+20\epsilon_+ > 0$ and a double
  eigenvalue $\alpha_5 = 2 + 2 \epsilon_+ > 0$ and $A {\bf 1}_6 = 0$. 
  $A_{\rm{red}}$ is now a symmetric $3 \times 3$ matrix with characteristic
  polynomial (see Maple calculations in Appendix \ref{app:thm:atbcde-BE})
  $$ \chi(t) = t^3-\frac{112}{5}t^2+\frac{144}{5}t $$
  in the normalized case and
  $$ \chi(t) = t^3-44t^2+72t $$
  in the non-normalized case. Similarly as before, this implies that 
  $A$ is positive semidefinite with one-dimensional kernel, that is,
  ${\mathcal K}_{G,x}(\infty) > 0$.
\end{proof}

\begin{rem}
  Alternatively, Theorem \ref{thm:atcde-BE} could be proved, in the
  non-normalized case, by employing the fact that the root of
  $\AT((1,2,3))$ is $S^1$-out regular. For the definition of this notion
  and the corresponding curvature calculation see \cite[Definition 1.5
  and Theorem 5.7]{CLP2016}.
\end{rem}

The above theorems imply that the infinite antitree $\AT((k))$ has
strictly positive Bakry-{\'E}mery curvature in all vertices. We
finally prove that there is no uniform positive lower curvature bound.

\begin{theorem} \label{thm:curv0BE}
  Let $G = \AT((k))$ be the infinite antitree with vertex set
  $V = \bigcup_{k=1}^\infty V_k$. Then we have both in the normalized
  and normalized setting
  $$ \inf_{x \in V} {\mathcal{K}}_{G,x}(\infty) = 0. $$
\end{theorem}

\begin{proof} Let us first consider the normalized setting. If we had
  $\inf_{x \in V} {\mathcal{K}}_{G,x}(\infty) = K > 0$, then the
  discrete Bonnet-Myers Theorem (Theorem \ref{thm:discBM-norm} of the
  Introduction) would imply that $G$ has bounded diameter, which is a
  contradiction. This argument does not work in the non-normalized
  setting. Let us now show in the non-normalized setting that
  $$ \lim_{n \to \infty, x \in V_n} {\mathcal{K}}_{G,x}(\infty) = 0. $$
  For $\delta > 0$, let $A(\delta,n) = 4(\Gamma_2(x) - \delta \Gamma(x))$
  for an arbitrary vertex $x \in V_{n+2}$, $n \in \IN$, with respect to the
  vertex order
  $$ B_2(x) = \{x\} \ \sqcup (V_{n+2} \backslash \{ x \}) \sqcup V_{n+3} \sqcup 
  V_{n+1} \sqcup V_{n+4} \sqcup V_{n}. $$
  The entries of $2\Gamma(x)$ in the non-normalized setting are given
  in \cite[(2.2)]{CLP2016}, and using this information, we see that
  that matrix $A(\delta,n)$ is of the following block structure
  $A(\delta,n) = (A_{ij}(\delta,n))_{1 \le i,j \le 6}$:
  \begin{eqnarray*}
    A_{11}(\delta,n) &=& (3n+5)(3n+8)-(6n+10)\delta, \\
    A_{12}(\delta,n) &=& (-3n-8+2\delta)J_{1,n+1}, \\
    A_{13}(\delta,n) &=& (-4n-12+2\delta)J_{1,n+3}, \\
    A_{14}(\delta,n) &=& (-4n-8+2\delta)J_{1,n+1}, \\
    A_{15}(\delta,n) &=& (n+3)J_{1,n+4}, \\
    A_{16}(\delta,n) &=& (n+1)J_{1,n},\\
    A_{22}(\delta,n) &=& (9n+18-2\delta)\Id_{n+1}-2J_{n+1}, \\
    A_{23}(\delta,n) &=& -2J_{n+1,n+3}, \\
    A_{24}(\delta,n) &=& -2J_{n+1,n+1}, \\
    A_{25}(\delta,n) &=& 0_{n+1,n+4}, \\
    A_{26}(\delta,n) &=& 0_{n+1,n}, \\
    A_{33}(\delta,n) &=& (8n+26-2\delta)\Id_{n+3}-2J_{n+3}, \\
    A_{34}(\delta,n) &=& 2J_{n+3,n+1}, \\
    A_{35}(\delta,n) &=& -2J_{n+3,n+4}, \\
    A_{36}(\delta,n) &=& 0_{n+3,n},
  \end{eqnarray*}
  \begin{eqnarray*}
    A_{44}(\delta,n) &=& (8n+6-2\delta)\Id_{n+1} - 2J_{n+1}, \\
    A_{45}(\delta,n) &=& 0_{n+1,n+4}, \\
    A_{46}(\delta,n) &=& -2J_{n+1,n}, \\
    A_{55}(\delta,n) &=& (n+3)\Id_{n+4}, \\
    A_{56}(\delta,n) &=& 0_{n+4,n}, \\
    A_{66}(\delta,n) &=& (n+1)\Id_n.
  \end{eqnarray*}
  Let $\delta > 0$. Let $\lambda_j(\delta,n)$, $j \in \{1,\dots, 5\}$
  be the eigenvalues of the $6 \times 6$ matrix
  $A(\delta,n)_{red}$. The characteristic polynomial of
  $A(\delta,n)_{red}$ is of the form
  $$ \chi_{\delta,n}(t) = t^6 - p_5(\delta,n) t^5 + p_4(\delta,n) t^4 
  - p_3(\delta,n) t^3 + p_2(\delta,n) t^2 - p_1(\delta,n) t,
  $$
  with polynomials $p_1,p_2,\dots,p_5$, and a Maple calculation shows that
  \begin{equation} \label{eq:p1form} 
  p_1(\delta,n) = - 240 \delta n^9 + q_8(\delta) n^8 + \dots + 
  q_1(\delta) n + q_0(\delta), 
  \end{equation}
  with polynomials $q_0,q_1,\dots,q_8$ (see Appendix
  \ref{app:thm:curv0BE}). By Vieta's formulas, we have
  $$ p_1(\delta,n) = \left( \prod_{j=1}^5 \lambda_j(\delta,n) \right), $$
  where $\lambda_j(\delta,n)$, $j=1,\dots,5$ are the eigenvalues (in
  ascending order) of $A(\delta,n)_{red}$ restricted to the orthogonal
  complement to the eigenvector ${\bf 1}_6$.  We conclude from
  \eqref{eq:p1form} that there exists $k_0 > 0$ with
  $p_1(\delta,n) < 0$ for all $n \ge n_0$, i.e.,
  $\lambda_1(\delta,n) < 0$.  Applying Lemma
  \ref{lem:usefact}, we conclude
  $$ (\widetilde w)^\top A(\delta,n) \widetilde w = w^\top 
  A(\delta,n)_{red} w = \lambda_1(\delta,n) \Vert w \Vert^2 < 0. $$
  This implies that ${\mathcal{K}}_{G,x}(\infty) \in (0,\delta)$ for every 
  $x \in V_{n+2}$ with $n \ge n_0$. 
\end{proof}

\section{Ollivier Ricci curvature of antitrees}
\label{sec:ollcalcs}

In this section, we calculate Ollivier-Ricci curvature for all
idlenesses $p \in [0,1]$ and the Lin-Lu-Yau curvature of all types of
edges in antitrees.

\begin{theorem}[Radial root-edges of an antitree]
Let $1\leq a\leq b\leq c,$ $\{x,y\}$ a radial root edge of the antitree $\mathcal{AT}((a,b,c)),$ that is $x\in V_{1}, y \in V_{2}.$ Then we have:
\begin{enumerate}[(a)]
\item
If $a = 1,$ 
\\
$\kappa_{p}(x,y) = \begin{cases} \frac{b-1}{b+c} + \frac{b+2c+1}{b+c}p &  \text{if $p \in [0,\frac{1}{b+c+1}]$,}\\
 \frac{b+1}{b+c}(1- p), &  \text{if $p\in [\frac{1}{b+c+1},1]$.}
\end{cases}$
\\
Therefore,
$$\kappa_{LLY}(x,y) = \frac{b+1}{b+c}.$$

\item
If $a\geq 3$ or $(a = 2 \:{\rm and}\: b<c),$
\\
$\kappa_{p}(x,y) = $
\\
$ \frac{1}{(a+b-1)(a+b+c-1)}\begin{cases} ((a+b-1)^{2}-c(a-1))+c(b+2a-2)p &  \text{if $p \in [0,\frac{1}{a+b+c}]$,}\\
 ((a+b)(a+b-1)-c(a-1))(1- p), &  \text{if $p\in [\frac{1}{a+b+c},1]$.}
\end{cases}$
\\
Therefore,
$$\kappa_{LLY}(x,y) = \frac{(a+b)(a+b-1)-c(a-1)}{(a+b-1)(a+b+c-1)}.$$
\item
If $a = 2, b = c,$ 
\\
$\kappa_{p}(x,y) = \begin{cases} \frac{b}{2b+1} + \frac{3b+2}{2b+1}p &  \text{if $p \in [0,\frac{1}{(2b+1)(b+1)1}]$,}\\
 \frac{b^{2}+b+1}{(2b+1)(b+1)}+\frac{b^{2}+2b}{(2b+1)(b+1)}p, &  \text{if $p\in [\frac{1}{(2b+1)(b+1)+1},\frac{1}{2(b+1)}]$,}\\
 \frac{b^{2}+2b+2}{(2b+1)(b+1)}(1- p), &  \text{if $p\in [\frac{1}{2(b+1)},1]$.}
\end{cases}$
\\
Therefore,
$$\kappa_{LLY}(x,y) = \frac{b^{2}+2b+2}{(2b+1)(b+1)}.$$
\end{enumerate}
\end{theorem}

\begin{proof}
\begin{enumerate}[(a)]
\item
Consider the following graph
$$\begin{tikzpicture}[x=1.5cm, y=1.5cm,
	vertex/.style={
		shape=circle, fill=black, inner sep=1.5pt	
	}
]

\node[vertex, label=below:$y'$] (1) at (0, 0) {};
\node[vertex, label=above:$v$] (4) at (0, 1) {};
\node[vertex, label=below:$z$] (2) at (1, 0) {};
\node[vertex, label=below:$x'$] (5) at (-1, 0) {};

\draw (1) -- (2);
\draw (1) -- (5);
\draw (1) -- (4);
\draw (5) -- (4);
\draw (2) -- (4);

\end{tikzpicture}
$$
with associated probability measures $\mu^{p}_{1},\mu^{p}_{2},$ defined as
$$\mu^{p}_{1}(x') = p,\:\: \mu^{p}_{1}(y') = \frac{1}{b}(1-p),\:\: \mu^{p}_{1}(v) = \frac{b-1}{b}(1-p),\:\: \mu^{p}_{1}(z) = 0,$$ 
$$\mu^{p}_{2}(x') = \frac{1}{b+c}(1-p),\:\: \mu^{p}_{2}(y') = p,\:\: \mu^{p}_{2}(v) = \frac{b-1}{b+c}(1-p),\:\: \mu^{p}_{2}(z) = \frac{c}{b+c}(1-p).$$ 
One can verify that, due to the high connectivity of $\mathcal{AT}((a,b,c)),$ we have $W_{1}(\mu^{p}_{x}, \mu^{p}_{y}) = W_{1}(\mu^{p}_{1}, \mu^{p}_{2}),$ where $x'$ represents the root $x$, $y'$ represents the vertex $y$, the vertex $v$ represents all neighbours of $y$ in $V_{2},$ and the vertex $z$ represents all vertices in $V_{3}.$  
\\
\\
Note that $\mu_{1}^{p}(x')< \mu_{2}^{p}(x')$ if and only if $p < \frac{1}{b+c+1}.$   We will distinguish the cases. 
\\
\\
Case $p < \frac{1}{b+c+1}:$
\\
Note that
$$\mu_1^{p}(x') < \mu_{2}^{p} (x'),\:\:\: \mu_1^{p}(z) < \mu_{2}^{p} (z),$$ 
$$\mu_1^{p}(y') > \mu_{2}^{p} (y'),\:\:\: \mu_1^{p}(w) > \mu_{2}^{p} (w).$$ 
Thus when transporting $\mu_{1}^{p}$ to $\mu_{2}^{p}$ the only vertices that gain mass are $x'$ and $z.$ Note further all this mass can be transported over a distance of $1$. Thus
\begin{align*}
W_{1}(\mu^{p}_{x}, \mu^{p}_{y}) & = W_{1}(\mu^{p}_{1}, \mu^{p}_{2})
\\
& \leq \mu_{2}^{p} (x')+ \mu_{2}^{p} (z) - \mu_{1}^{p} (x') - \mu_{1}^{p} (z) 
\\
& = \frac{c+1}{b+c} -\frac{b+2c+1}{b+c}p .
\end{align*} 
We verify that this is in fact equality by constructing the following $\phi\in 1-$Lip,
$$\phi(x') = 0, \phi(y') = 1, \phi(w) = 1, \phi(z) = 0. $$
Then, by Theorem \ref{Kantorovich},
$$W_{1}(\mu^{p}_{x}, \mu^{p}_{y}) = W_{1}(\mu^{p}_{1}, \mu^{p}_{2})\geq  \sum_{v}\phi(v)(\mu^{p}_{1}(v)-\mu^{p}_{2}(v)) = \frac{c+1}{b+c} -\frac{b+2c+1}{b+c}p.$$
Therefore 
$$W_{1}(\mu^{p}_{x}, \mu^{p}_{y}) = \frac{c+1}{b+c} -\frac{b+2c+1}{b+c}p.$$
and 
\begin{equation}\label{eq1}
\kappa_{p}(x,y) = \frac{b-1}{b+c} + \frac{b+2c+1}{b+c}p,
\end{equation}
for $p \in
[0,\frac{1}{b+c+1}).$ By continuity of
$p\mapsto\kappa_{p}(x,y)$ this also holds for $p = \frac{1}{b+c+1}.$
\\
\\
Case $p \geq \frac{1}{b+c+1}:$
\\
By \cite[Theorem 4.4]{BCLMP}, $\kappa_{p}(x,y) =
\frac{b+c+1}{b+c}\kappa_{\frac{1}{b+c+1}}(1- p)$ for $p\in
[\frac{1}{b+c+1},1].$ Thus
$$\kappa_{p}(x,y) = \begin{cases} \frac{b-1}{b+c} + \frac{b+2c+1}{b+c}p &  \text{if $p \in [0,\frac{1}{b+c+1}]$,}\\
 \frac{b+c+1}{b+c}\kappa_{\frac{1}{b+c+1}}(1- p), &  \text{if $p\in [\frac{1}{b+c+1},1]$.}
\end{cases}$$
Therefore it only remains to show that $\frac{b+c+1}{b+c}\kappa_{\frac{1}{b+c+1}} = \frac{b+1}{b+c}.$
\\
\\
We have, using \eqref{eq1},
\begin{align*}
\frac{b+c+1}{b+c}\kappa_{\frac{1}{b+c+1}} & = \frac{b+c+1}{b+c}\left(\frac{b-1}{b+c} + \frac{b+2c+1}{b+c}\frac{1}{b+c+1}\right)
\\
& = \frac{b+1}{b+c}.
\end{align*} 

\item
Similar to above we consider the simplified graph representing $\mathcal{AT}((a,b,c)),$
$$\begin{tikzpicture}[x=1.5cm, y=1.5cm,
	vertex/.style={
		shape=circle, fill=black, inner sep=1.5pt	
	}
]

\node[vertex, label=below:$y'$] (1) at (0, 0) {};
\node[vertex, label=above:$v$] (4) at (0, 1) {};
\node[vertex, label=above:$u$] (3) at (-1, 1) {};
\node[vertex, label=below:$z$] (2) at (1, 0) {};
\node[vertex, label=below:$x'$] (5) at (-1, 0) {};

\draw (1) -- (2);
\draw (1) -- (5);
\draw (1) -- (4);
\draw (5) -- (4);
\draw (2) -- (4);
\draw (5) -- (3);
\draw (3) -- (4);
\draw (3) -- (1);

\end{tikzpicture}
$$

with associated probability measures $\mu^{p}_{1},\mu^{p}_{2},$ defined as
$$\mu^{p}_{1}(x') = p,\:\: \mu^{p}_{1}(y') = \frac{1}{a+b-1}(1-p),\:\: \mu^{p}_{1}(u) = \frac{a-1}{a+b-1}(1-p),$$ 
$$\mu^{p}_{1}(v) = \frac{b-1}{a+b-1}(1-p),\:\:  \mu^{p}_{1}(z) = 0,$$ 
$$\mu^{p}_{2}(x') = \frac{1}{a+b+c-1}(1-p),\:\: \mu^{p}_{2}(y') = p,\:\: \mu^{p}_{2}(u) = \frac{a-1}{a+b+c-1}(1-p),$$
$$ \mu^{p}_{2}(v) = \frac{b-1}{a+b+c-1}(1-p),\:\: \mu^{p}_{2}(z) = \frac{c}{a+b+c-1}(1-p).$$ 
Again, one can verify that, due to the high connectivity of $\mathcal{AT}((a,b,c)),$ we have $W_{1}(\mu^{p}_{x}, \mu^{p}_{y}) = W_{1}(\mu^{p}_{1}, \mu^{p}_{2}),$ where $x'$ represents the root $x$, $y'$ represents the vertex $y$, the vertex $u$ represents all neighbours of $x$ in $V_{1},$ the vertex $v$ represents all neighbours of $y$ in $V_{2},$ and the vertex $z$ represents all vertices in $V_{3}.$  
\\
\\
Let $p\in (0,\frac{1}{a+b+c}).$ One can check that 
$$\mu_{1}^{p}(x')<\mu_{2}^{p}(x'),\:\: \mu_{1}^{p}(z)<\mu_{2}^{p}(z),$$
$$\mu_{1}^{p}(y')>\mu_{2}^{p}(y'),\:\: \mu_{1}^{p}(u)>\mu_{2}^{p}(u),\:\: \mu_{1}^{p}(v)>\mu_{2}^{p}(v).$$
Thus the vertices $x'$ and $z$ must gain mass and the vertices $u,v$ and $y$ must lose mass. We now show that some mass must be transported from $u$ to $z$. Suppose that no mass is moved from $u$ to $v$. Then the mass available to move from $v$ and $y'$ will be sufficient when moved to $z$. Therefore
$$\mu_{1}^{p}(y')+\mu_{1}^{p}(v)-\mu_{2}^{p}(y')-\mu_{2}^{p}(v)  \geq \mu_{2}^{p}(z) - \mu_{1}^{p}(z).$$
Substituting in the values of the measures and rearranging gives $p\leq \frac{a+b+c-ac-1}{(a+b)(a+b-1)+bc}\leq 0,$ a contradiction. Therefore some mass must be transported from $u$ to $z$ over a distance of $2$ and all other mass can be transported over a distance of $1$. 
\\
\\
Thus
\begin{align*}
W_{1}(\mu^{p}_{x}, \mu^{p}_{y}) = & W_{1}(\mu^{p}_{1}, \mu^{p}_{2})
\\
 \leq & (\mu_{2}^{p}(x)-\mu_{1}^{p}(x))+2(\mu^{p}_{1}(u)-\mu^{p}_{2}(u)-(\mu_{2}^{p}(x)-\mu_{1}^{p}(x)))
\\ 
& +(\mu^{p}_{1}(y')+\mu^{p}_{1}(v)-\mu^{p}_{2}(y')-\mu^{p}_{2}(v))
\\
= & (1-p)\left(\frac{a-1}{a+b-1}+\frac{c+1-a}{a+b+c-1}\right).
\end{align*}
We verify that this is in fact equality by constructing the following $\phi\in 1-$Lip,
$$\phi(x') = 0, \phi(y') = 0, \phi(u) = 1, \phi(v) = 0, \phi(z) =  -1. $$
Therefore 
\begin{align*}
\kappa_{p}(x,y) & = 1 - (1-p)\left(\frac{a-1}{a+b-1}+\frac{c+1-a}{a+b+c-1}\right) 
\\
& = \frac{((a+b-1)^{2}-c(a-1))+(bc+2c(a-1))p}{(a+b-1)(a+b+c-1)},
\end{align*}
for $p \in (0,\frac{1}{a+b+c}).$ 
\\
\\
As before, by \cite[Theorem 4.4]{BCLMP}, $\kappa_{p}(x,y) = \frac{a+b+c}{a+b+c-1}\kappa_{\frac{1}{a+b+c}}(1- p)$  for $p\in [\frac{1}{a+b+c},1].$ Therefore
$$\frac{a+b+c}{a+b+c-1}\kappa_{\frac{1}{a+b+c}}  = \frac{(a+b)(a+b-1)-c(a-1)}{(a+b-1)(a+b+c-1)},$$
thus completing the proof.

\item
As in part (b)  we consider the simplified graph representing $\mathcal{AT}((a,b,c)),$
$$\begin{tikzpicture}[x=1.5cm, y=1.5cm,
	vertex/.style={
		shape=circle, fill=black, inner sep=1.5pt	
	}
]

\node[vertex, label=below:$y'$] (1) at (0, 0) {};
\node[vertex, label=above:$v$] (4) at (0, 1) {};
\node[vertex, label=above:$u$] (3) at (-1, 1) {};
\node[vertex, label=below:$z$] (2) at (1, 0) {};
\node[vertex, label=below:$x'$] (5) at (-1, 0) {};

\draw (1) -- (2);
\draw (1) -- (5);
\draw (1) -- (4);
\draw (5) -- (4);
\draw (2) -- (4);
\draw (5) -- (3);
\draw (3) -- (4);
\draw (3) -- (1);

\end{tikzpicture}
$$

with the same associated probability measures $\mu^{p}_{1},\mu^{p}_{2},$ defined as
$$\mu^{p}_{1}(x') = p,\:\: \mu^{p}_{1}(y') = \frac{1}{a+b-1}(1-p),\:\: \mu^{p}_{1}(u) = \frac{a-1}{a+b-1}(1-p),$$ 
$$\mu^{p}_{1}(v) = \frac{b-1}{a+b-1}(1-p),\:\:  \mu^{p}_{1}(z) = 0,$$ 
$$\mu^{p}_{2}(x') = \frac{1}{a+b+c-1}(1-p),\:\: \mu^{p}_{2}(y') = p,\:\: \mu^{p}_{2}(u) = \frac{a-1}{a+b+c-1}(1-p),$$
$$ \mu^{p}_{2}(v) = \frac{b-1}{a+b+c-1}(1-p),\:\: \mu^{p}_{2}(z) = \frac{c}{a+b+c-1}(1-p).$$ 
Again, one can verify that, due to the high connectivity of $\mathcal{AT}((a,b,c)),$ we have $W_{1}(\mu^{p}_{x}, \mu^{p}_{y}) = W_{1}(\mu^{p}_{1}, \mu^{p}_{2}),$ where $x'$ represents the root $x$, $y'$ represents the vertex $y$, the vertex $u$ represents all neighbours of $x$ in $V_{1},$ the vertex $v$ represents all neighbours of $y$ in $V_{2},$ and the vertex $z$ represents all vertices in $V_{3}.$ 
\\
\\
We will distinguish the cases.
\\
\\
Case $p \in (0,\frac{1}{(2b+1)(b+1)1}):$
\\
One can check that 
$$\mu_{1}^{p}(x')<\mu_{2}^{p}(x'),\:\: \mu_{1}^{p}(z)<\mu_{2}^{p}(z),$$
$$\mu_{1}^{p}(y')>\mu_{2}^{p}(y'),\:\: \mu_{1}^{p}(u)>\mu_{2}^{p}(u),\:\: \mu_{1}^{p}(v)>\mu_{2}^{p}(v),$$
and 
$$\mu_{1}^{p}(y')+\mu_{1}^{p}(v)-\mu_{2}^{p}(y')-\mu_{2}^{p}(v)  \geq \mu_{2}^{p}(z) - \mu_{1}^{p}(z).$$
Thus the vertices $x'$ and $z$ must gain mass and the vertices $u,v$ and $y.$ must lose mass and it is possible for all mass to be moved over a distance of $1$. 
\\
\\
Thus
\begin{align*}
W_{1}(\mu^{p}_{x}, \mu^{p}_{y}) & = W_{1}(\mu^{p}_{1}, \mu^{p}_{2})
\\
& \leq \mu_{2}^{p} (x')+ \mu_{2}^{p} (z) - \mu_{1}^{p} (x') - \mu_{1}^{p} (z) 
\\
& = \frac{b+1}{2b+1} - \frac{3b+2}{2b+1}p.
\end{align*} 
We verify that this is in fact equality by constructing the following $\phi\in 1-$Lip,
$$\phi(x') = -1, \phi(y') = 0, \phi(u) = 0, \phi(v) = 0, \phi(z) =  -1. $$
Therefore
$$\kappa_{p}(x,y) = \frac{b}{2b+1} + \frac{3b+2}{2b+1}p.$$
Case $p\in (\frac{1}{(2b+1)(b+1)+1},\frac{1}{2(b+1)}):$
\\
One can check that we still have 
$$\mu_{1}^{p}(x')<\mu_{2}^{p}(x'),\:\: \mu_{1}^{p}(z)<\mu_{2}^{p}(z),$$
$$\mu_{1}^{p}(y')>\mu_{2}^{p}(y'),\:\: \mu_{1}^{p}(u)>\mu_{2}^{p}(u),\:\: \mu_{1}^{p}(v)>\mu_{2}^{p}(v)$$
However we now have 
$$\mu_{1}^{p}(y')+\mu_{1}^{p}(v)-\mu_{2}^{p}(y')-\mu_{2}^{p}(v)  \leq \mu_{2}^{p}(z) - \mu_{1}^{p}(z).$$
Thus, as in part (b), some mass must be transported from $u$ to $z$ over a distance of $2$ and all other mass can be transported over a distance of $1$. 
\\
\\
Therefore
\begin{align*}
W_{1}(\mu^{p}_{x}, \mu^{p}_{y}) = & W_{1}(\mu^{p}_{1}, \mu^{p}_{2})
\\
 \leq & (\mu_{2}^{p}(x)-\mu_{1}^{p}(x))+2(\mu^{p}_{1}(u)-\mu^{p}_{2}(u)-(\mu_{2}^{p}(x)-\mu_{1}^{p}(x)))
\\ 
& +(\mu^{p}_{1}(y')+\mu^{p}_{1}(v)-\mu^{p}_{2}(y')-\mu^{p}_{2}(v))
\\
= & (1-p)\left(\frac{1}{b+1}+\frac{b-1}{2b+1}\right).
\end{align*}
We verify that this is in fact equality by constructing the following $\phi\in 1-$Lip,
$$\phi(x') = 0, \phi(y') = 0, \phi(u) = 1, \phi(v) = 0, \phi(z) =  -1. $$
Therefore 
$$\kappa_{p} (x,y) = \frac{b^{2}+b+1}{(2b+1)(b+1)}+\frac{b^{2}+2b}{(2b+1)(b+1)}p.$$
Case $p\in (\frac{1}{2(b+1)},1):$
As before, by \cite[Theorem 4.4]{BCLMP}, $\kappa_{p}(x,y) = \frac{2(b+1)}{2b+1}\kappa_{\frac{1}{2(b+1)}}(1- p)$  for $p\in [\frac{1}{2(b+1)},1].$ Thus
$$\frac{2(b+1)}{2b+1}\kappa_{\frac{1}{2(b+1)}} = \frac{b^{2}+2b+2}{(2b+1)(b+1)},$$
thus completing the proof.
\end{enumerate}
\end{proof}

\begin{theorem}[Inner radial edges of an antitree]
Let $1\leq a\leq b\leq c \leq d,$ $\{x,y\}$ an inner radial edge of the antitree $\mathcal{AT}((a,b,c,d)),$ that is $x\in V_{2}, y \in V_{3}.$ Then we have:
$$\kappa_{p}(x,y) = \left(\frac{2b+c-1}{b+c+d-1} - \frac{2a+b-1}{a+b+c-1}\right)(1-p).$$
\end{theorem}

\begin{proof}
We first calculate $\kappa_{0}(x,y).$ We consider the simplified graph representing $\mathcal{AT}((a,b,c,d)),$
$$\begin{tikzpicture}[x=1.5cm, y=1.5cm,
	vertex/.style={
		shape=circle, fill=black, inner sep=1.5pt	
	}
]

\node[vertex, label=below:$y'$] (1) at (0, 0) {};
\node[vertex, label=above:$v$] (4) at (0, 1) {};
\node[vertex, label=above:$u$] (3) at (-1, 1) {};
\node[vertex, label=below:$z$] (2) at (1, 0) {};
\node[vertex, label=below:$x'$] (5) at (-1, 0) {};
\node[vertex, label=below:$w$] (6) at (-2, 0) {};

\draw (1) -- (2);
\draw (1) -- (5);
\draw (1) -- (4);
\draw (5) -- (4);
\draw (2) -- (4);
\draw (5) -- (3);
\draw (3) -- (4);
\draw (3) -- (1);
\draw (5) -- (6);
\draw (3) -- (6);

\end{tikzpicture}
$$
with the associated probability measures $\mu_{1},\mu_{2},$ defined as
$$\mu_{1}(x') = 0,\:\: \mu_{1}(y') = \frac{1}{a+b+c-1},\:\: \mu_{1}(w) = \frac{a}{a+b+c-1},$$ 
$$\mu_{1}(u) = \frac{b-1}{a+b+c-1},\:\: \mu_{1}(v) = \frac{c-1}{a+b+c-1},\:\:  \mu_{1}(z) = 0,$$ 
$$\mu_{2}(x') = \frac{1}{b+c+d-1},\:\: \mu_{2}(y') = 0,\:\: \mu_{2}(w) = 0,$$
$$\mu_{2}(u) = \frac{b-1}{b+c+d-1},\:\: \mu_{2}(v) = \frac{c-1}{b+c+d-1},\:\: \mu_{2}(z) = \frac{d}{b+c+d-1}.$$ 
Again, one can verify that, due to the high connectivity of $\mathcal{AT}((a,b,c,d)),$ we have $W_{1}(\mu^{0}_{x}, \mu^{0}_{y}) = W_{1}(\mu_{1}, \mu_{2}),$ where $x'$ represents the vertex $x$, $y'$ represents the vertex $y$, the vertex $w$ represents all the vertices in $V_{1},$ the vertex $u$ represents all neighbours of $x$ in $V_{2},$ the vertex $v$ represents all neighbours of $y$ in $V_{3},$ and the vertex $z$ represents all vertices in $V_{4}.$ 
\\
\\
Observe that 
$$\mu_{1}(x') < \mu_{2}(x'),\:\: \mu_{1}(z) < \mu_{2}(z),\:\: \mu_{1}(u) < \mu_{2}(u),\:\: \mu_{1}(v) < \mu_{2}(v),$$
$$\mu_{1}(y') > \mu_{2}(y'),\:\: \mu_{1}(w) > \mu_{2}(w).$$
Therefore the only vertices that gain mass are $x'$ and $z.$ Now, $\mu_{1}(w)-\mu_{2}(w) = \frac{a}{a+b+c-1} \geq \frac{1}{b+c+d-1} = \mu_{2}(x') - \mu_{1}(x'),$ and so it is possible for $x'$ to receive all of its needed mass from $w.$ If we do this plan and send all other surplus mass to the vertex $z$ we obtain 
\begin{align*}
W_{1}(\mu^{p}_{x}, \mu^{p}_{y}) = & W_{1}(\mu^{p}_{1}, \mu^{p}_{2})
\\
 \leq & (\mu_{2}(x') - \mu_{1}(x')) + 3(\mu_{1}(w)-[\mu_{2}(x') - \mu_{1}(x')]-\mu_{2}(w)) + 2(\mu_{1}(u)-\mu_{2}(u))
\\ 
& +(\mu_{1}(v)-\mu_{2}(v))+(\mu_{1}(y')-\mu_{2}(y'))
\\
= & \frac{3a+2b+c-2}{a+b+c-1} - \frac{2b+c-1}{b+c+d-1}.
\end{align*}
We verify that this is in fact equality by constructing the following $\phi\in 1-$Lip,
$$\phi(w) = 3,\:\:\phi(x') = 2, \:\: \phi(u) = 2,\:\: \phi(y') = 1,\:\: \phi(v) = 1,\:\: \phi(z) =  0. $$
Thus 
$$\kappa_{0}(x,y) = \frac{2b+c-1}{b+c+d-1} - \frac{2a+b-1}{a+b+c-1}.$$
Observe that $\phi(x')-\phi(y') = 1$ and thus, by \cite[Lemma 4.2]{BCLMP}, we have that $p\mapsto \kappa_{p}(x,y)$ is linear. Since $\kappa_{1}(x,y) = 0,$ this gives
$$\kappa_{p}(x,y)\left(\frac{2b+c-1}{b+c+d-1} - \frac{2a+b-1}{a+b+c-1}\right)(1-p).$$
\end{proof}

\begin{theorem}[Spherical root edges of an antitree]\label{spher}
Let $2\leq a\leq b,$ $\{x,y\}$ a spherical root edge of the antitree $\mathcal{AT}((a,b)),$ that is $x,y\in V_{1}.$ Then 
$$\kappa_{p}(x,y) = \begin{cases} \frac{a+b-2}{a+b-1} + \frac{a+b}{a+b-1}p &  \text{if $p \in [0,\frac{1}{a+b}]$,}\\
 \frac{a+b}{a+b-1}(1- p), &  \text{if $p\in [\frac{1}{a+b},1]$.}
\end{cases}$$
\end{theorem}

\begin{proof}
Since $d_{x} = d_{y},$ by \cite[ Theorem 5.3]{BCLMP}, we have 
$$\kappa_{ p}(x,y) = \begin{cases} ((a+b-1)\kappa_{LLY}(x,y)-(a+b)\kappa_{0}(x,y)) p + \kappa_{0}(x,y), &   \text{if $p\in [0,\frac{1}{a+b}]$,}\\
(1- p)\kappa_{LLY}(x,y), &   \text{if $p \in [\frac{1}{a+b},1]$.}
\end{cases} $$
Therefore we will calculate $\kappa_{p}(x,y)$ for $p=0$ and $p=\frac{1}{a+b}.$ 
\\
\\
Observe that $\mu_{x}^{0}(y) = \frac{1}{a+b-1}$ and $0$ otherwise, and $\mu_{y}^{0}(x) = \frac{1}{a+b-1}$ and $0$ otherwise. Thus we have
$$W_{1}(\mu_{x}^{0},\mu_{y}^{0}) = \frac{1}{a+b-1},$$
and so
$$\kappa_{0}(x,y) = \frac{a+b-2}{a+b-1}.$$ 
Note that $$\mu_{x}^{\frac{1}{a+b}} \equiv \mu_{y}^{\frac{1}{a+b}} ,$$ so 
$$\kappa_{LLY}(x,y) = \frac{a+b}{a+b-1} \kappa_{\frac{1}{a+b}}(x,y) = \frac{a+b}{a+b-1}.$$
Substituting these values in to the above formula completes the proof.
\end{proof}

\begin{theorem}[Spherical inner edges of an antitree]
Let $1\leq a\leq b\leq c,$ $\{x,y\}$ a spherical inner edge of the antitree $\mathcal{AT}((a,b,c)),$ that is $x,y\in V_{2}.$ Then 
$$\kappa_{p}(x,y) = \begin{cases} \frac{a+b+c-2}{a+b+c-1} + \frac{a+b+c}{a+b+c-1}p &  \text{if $p \in [0,\frac{1}{a+b+c}]$,}\\
 \frac{a+b+c}{a+b+c-1}(1- p), &  \text{if $p\in [\frac{1}{a+b+c},1]$.}
\end{cases}$$
\end{theorem}

\begin{proof}
The proofs follows in the same way as in the proof of Theorem \ref{spher}.
\end{proof}

\newpage

\begin{appendices}

\section{Maple Calculations for Theorem \ref{thm:atabcde-BE}}
\label{app:thm:atabcde-BE}

In the \emph{normalized case}, the Maple code to construct the matrix
$A_{\rm{red}} = 4 \mu_x^2 \Gamma_{2,\rm{red}}(x)$ for $x \in V_3 \cong K_c$ of
${\mathcal{AT}}((a,b,c,d,e))$ is the following:

\begin{figure}[h]
\includegraphics[width=\textwidth]{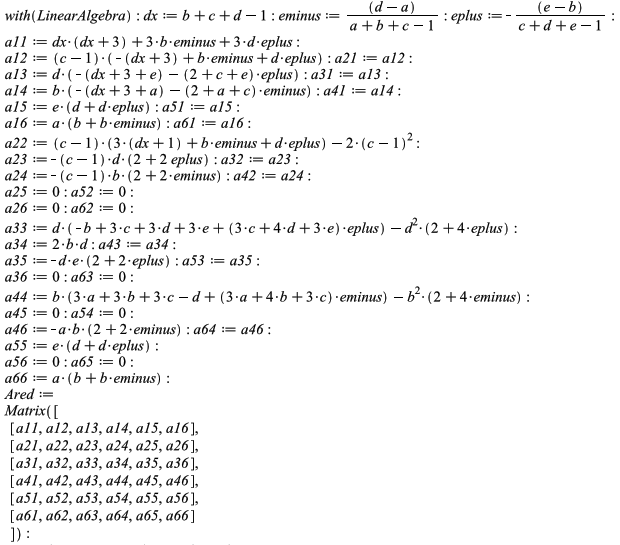}
\caption{Maple construction of $A_{\rm{red}}$ in the normalized case}
\label{fig:construct_Ared}
\end{figure}
 
For the generation of the coefficients of the charactestic polynomial
$\chi_n(t)$ of $A_{\rm{red}}$ for $a=n, b=n+1, c=n+2, d=n+3, e=n+4$,
see Figure \ref{fig:coeff-charpoly-norm}. Note that there are no
negative coefficients in the polynomials
$p_1(n), p_2(n), p_3(n), p_4(n)$ and $p_5(n)$.

\begin{figure}
\includegraphics[width=\textwidth]{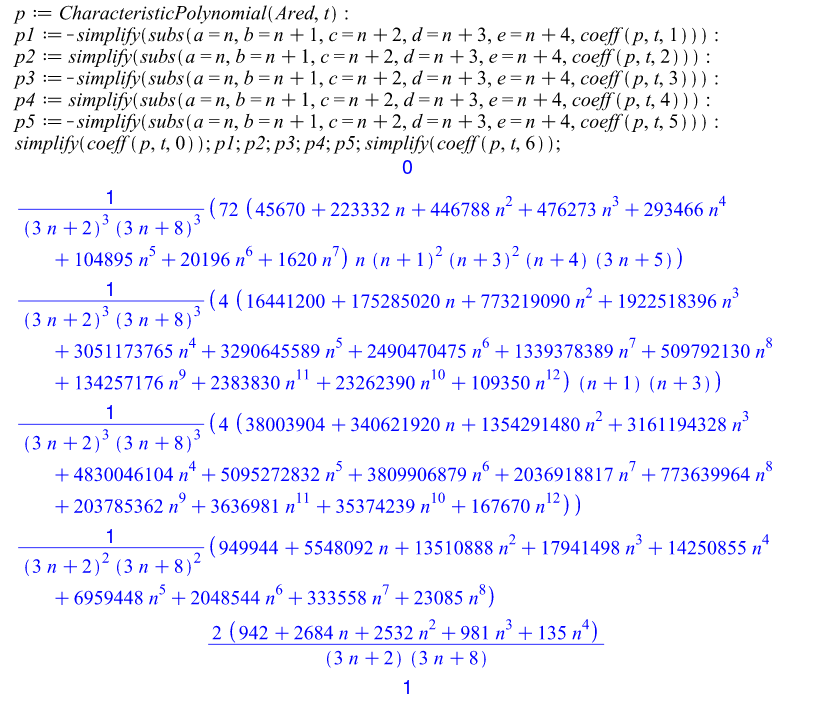}
\caption{Coefficients of $\chi_n(t) = \det(t {\rm{Id}}_6 - A_{\rm{red}})$, normalized case} \label{fig:coeff-charpoly-norm}
\end{figure}

The only modification of the above code in the \emph{non-normalized
  case} is to set the variables {\texttt{eminus}} and {\texttt{eplus}}
equal to $0$. The coefficients of $\chi_n(t)$ for
$a=n, b=n+1, c=n+2, d=n+3, e=n+4$ are given in Figure
\ref{fig:coeff-charpoly-nonnorm}. Again, all coefficients of $p_j(n)$,
$j=1,2,3,4,5$, are non-negative.

\begin{figure}
\includegraphics[width=\textwidth]{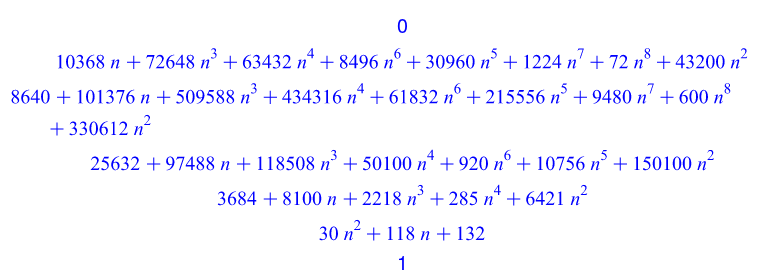}
\caption{Coefficients of $\chi_n(t) = \det(t {\rm{Id}}_6 - A_{\rm{red}})$, non-normalized case} \label{fig:coeff-charpoly-nonnorm}
\end{figure}

\section{Maple Calculations for Theorems \ref{thm:atbcde-BE} and \ref{thm:atcde-BE}}
\label{app:thm:atbcde-BE}

For the Maple calculations needed for the proofs of these theorems,
the code of Figure \ref{fig:construct_Ared} is used again, followed by
the code in Figure \ref{fig:otherthms} (in the \emph{normalized}
case). The reduced matrices $A_{\rm{red}}$ are here of dimension $5$
and $3$, respectively, and they can be extracted from the original
$6 \times 6$ matrix as submatrices with specific choices for
$a,b,c,d,e$. The crucial observation here is that the coefficients of
the respective characteristic polynomials of degree $5$ and $3$ are
alternating, guaranteeing that all non-zero roots are strictly
positive.
\begin{figure}
\includegraphics[width=\textwidth]{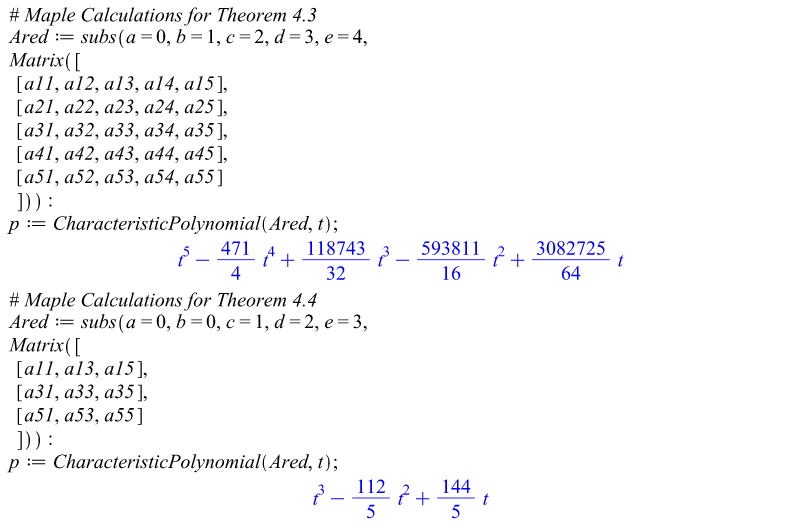}
\caption{Calculation of $\chi(t) = \det(t {\rm{Id}} - A_{\rm{red}})$ for Theorems \ref{thm:atbcde-BE} amd \ref{thm:atcde-BE}, normalized case} \label{fig:otherthms}
\end{figure} 
As before, the \emph{non-normalized case} is treated analogously with the
small modification to set the variables {\texttt{eminus}} and
{\texttt{eplus}} equal to $0$. This leads again to characteristic
polynomials with alternating coefficients, given in the proofs of the
theorems as
$$ \chi(t) = t^5 - 132 t^4 + 3684 t^3 - 25632 t^2 + 8640 t $$
and
$$ \chi(t) = t^3-44t^2+72t. $$

\section{Maple Calculations for Theorem \ref{thm:curv0BE}}
\label{app:thm:curv0BE}

Using the information about $(A_{ij}(\delta,n))$ in the proof of
Theorem \ref{thm:curv0BE}, the Maple code to calculate the relevant
polynomial $p_1(\delta,n)$ is given in Figure \ref{fig:curv0}.
\begin{figure}
\includegraphics[width=\textwidth]{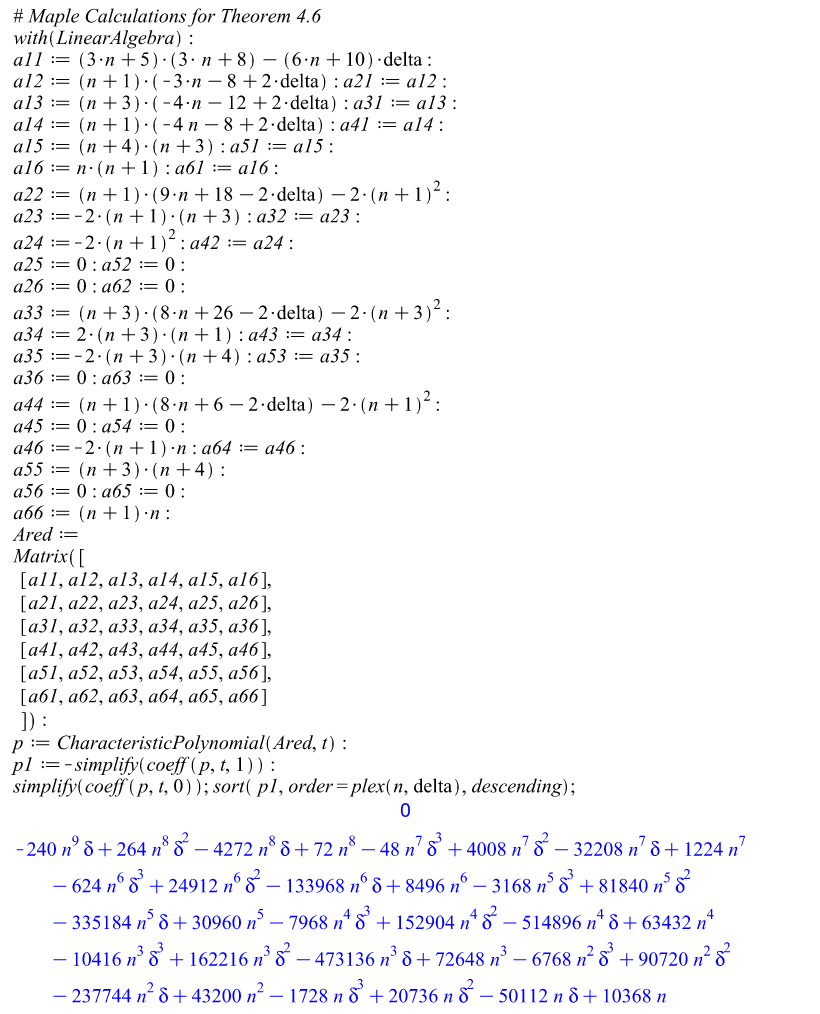}
\caption{Calculation of $p_1(\delta,n)$ in the proof of Theorem
  \ref{thm:curv0BE}} \label{fig:curv0}
\end{figure}

\end{appendices}

\end{document}